\documentclass[11pt,centertags]{amsart}
\usepackage{marginnote}

\usepackage{hyperref}
\usepackage{amssymb}            % extra common ams symbols -- not autoloaded
\usepackage{enumerate}          % improved enumeration environment control
\usepackage{fancyhdr}           % super headers --recommended
\usepackage{url}                % appropriately display url's
\usepackage{yhmath}
\usepackage{color,todonotes}
\usepackage{graphicx}

\numberwithin{equation}{section}

\newcommand{\comm}[1]{}
\newtheorem{theorem}{Theorem}
\newtheorem{definition}[theorem]{Definition}
\newtheorem{lemma}[theorem]{Lemma}
\newtheorem{question}[theorem]{Question}
\newtheorem{remark}[theorem]{Remark}
\newtheorem{proposition}[theorem]{Proposition}
\newtheorem{corollary}[theorem]{Corollary}

\numberwithin{theorem}{section}

\theoremstyle{remark}

\DeclareMathOperator{\Sp}{Sp}

\DeclareMathOperator{\SU}{SU}

\DeclareMathOperator{\Nbhd}{Nbhd}

\newcommand{\interior}[1]{%
  {\kern0pt#1}^{\mathrm{o}}%
}

 \newcounter{case}
 
 \renewcommand{\thecase}{\arabic{case}}

\newcommand{\Ga}{\Gamma}
\newcommand{\ga}{\gamma}

\newcommand{\op}[1]{\operatorname{#1}}

\providecommand{\to}{\longrightarrow }

		\renewcommand{\bf}{\bfseries}

		\renewcommand{\bar}{\overline}
		
\newcommand{\cout}[1]{}
\begin{document}

\author[Thang Nguyen]{Thang Nguyen}
%\thanks{$\dagger$}
\author[Shi Wang]{Shi Wang}
%\thanks{$\dagger$}
%\author[]{} % For second author, keep adding authors this way
% One \thanks command per author

\title[MLS for relatively hyperbolic groups]{Marked length spectrum rigidity for relatively hyperbolic groups}

\address{Florida State University, Tallahassee, Florida, USA} %one \address command per author
\email{thang.q.nguyen7@gmail.com} % one \email command per author

\address{ShanghaiTech University, Shanghai, China} %one \address command per author
\email{shiwang.math@gmail.com}
%\curraddr{}
%\urladdr{} % use \textasciitilde instead of ~ in URL
%\dedicatory{}
%\date{\today} % not standard to put at bottom
%\translator{}
%\keywords{}

\begin{abstract}
We consider a coarse version of the marked length spectrum rigidity: given a group with two left invariant metrics, if the marked length spectrum (the translation length function) under the two metrics are the same, then the two metrics are uniformly close. We prove the rigidity theorem for relatively hyperbolic groups. This generalizes a result of Fujiwara \cite{Fujiwara15}.
\end{abstract}

\maketitle

%\markboth{Short author names}{Short title} % used for subsequent
%pages. less desirable alternative

\thispagestyle{empty} % turn off page numbering on title page

% Text begins Below
%**************************************************************************
\section{Introduction}
Given a closed Riemannian manifold $(M,g)$, if $g$ has negative curvature, then each free homotopy class $c$ associates a unique closed geodesic $\ga_c$ on $M$. Denote $\mathcal C$ the set of all free homotopy classes, or equivalently the set of conjugacy classes of $\pi_1(M)$. The function 
\[\ell_g:\mathcal C\rightarrow \mathbb R^{\geq 0}\]
given by $\ell_g(c):=L(\ga_c)$ is called the \emph{marked length spectrum} of $(M,g)$, where $L$ denotes the length of a curve. The definition extends naturally to manifolds of nonpositive curvature, in which case the closed geodesic might not be unique, but its length is uniquely determined by the homotopy class.

The well-known marked length spectrum rigidity conjecture \cite{Burns-Katok1985} states that if two closed negatively curved Riemannian manifolds have the same marked length spectrum, then they are isometric. This is known to be true for surfaces by the result of Otal \cite{Otal} and Croke \cite{Croke} independently. When one of the Riemannian manifolds is rank one locally symmetric, Hamenst\"adt \cite{Hamenstaedt99} proved the conjecture using the minimal entropy rigidity theorem of Besson-Courtois-Gallot \cite{BCG95}. More recently, Guillarmou and Lefeuvre \cite{Guillarmou19} showed that the conjecture holds if the two metrics are close enough. However, the conjecture remains open in general. 

The main purpose of this draft is to consider a ``coarse'' version of the marked length spectrum rigidity problem. Let $\Ga$ be a finitely generated group acting isometrically on a metric space $(X,d)$. A natural generalization of the marked length spectrum is given by the following.

\begin{definition}
Let $\Gamma$ be a group acting on a metric space $(X,d)$ by isometries. For every $\gamma\in \Gamma$, the \emph{translation length} of $\gamma$ is defined as
\[|\gamma|_\infty=\lim_{n\to +\infty}\frac{d(x,\gamma^n x)}{n},\]
for any (and thus all) $x\in X$. Similarly, we call the function
\[\ell_d:\Ga\rightarrow \mathbb [0,+\infty)\]
given by $\ell_d(\ga)=|\ga|_\infty$ the \emph{marked length spectrum} of the $\Ga$-action on $(X,d)$.
\end{definition}

We note that, when $(X,d)$ is a CAT(0) space, the translation length defined above coincides with the usual definition of translation length of an isometry $\ga$, which is defined as
\[\ell(\ga)=\inf_{x\in X}d(x,\ga x).\]
In Furman \cite{Furman02}, this notion is called \emph{stable length} as it coincides with the usual stable length when $X$ is a Cayley graph of $\Gamma$. Here we choose to use the terminology translation length as in Fujiwara's \cite{Fujiwara15}.

The translation length only depends on the conjugacy class of $\ga$. In the case $\Ga=\pi_1(M)$ where $M$ is a nonpositively curved Riemannian manifold, the above defined marked length spectrum for the action of $\Ga$ on the Riemannian universal cover $\widetilde M$ coincides with the classical marked length spectrum for $(M,g)$.

Unlike the Riemannian setting, the large scale geometry of the metric spaces is often considered in geometric group theory. Thus, we may naturally consider the following coarse marked length spectrum rigidity problem.

\begin{question}\label{ques:main} Let $\Gamma$ be a finitely generated group, and let $X$ and $Y$ be metric spaces. We assume that $\Gamma$ acts on $X$ and $Y$ cocompactly (or more generally coboundedly) by isometries. Suppose the two actions have the same marked length spectrum, is it true that $X$ and $Y$ must be roughly isometric, namely there exists $f:X\to Y$ and there exists $C>0$ such that $f$ is a $(1,C)$-quasi-isometry?
\end{question}

We will say two metrics are \emph{roughly equal} if they are $(1,C)$-quasi-isometric for some $C\ge 0$. This is also called ``an almost isometry'' in the literature and certain rigidity results have been studied in \cite{Kar-Lafont-Schmidt, Lafont-Schmidt-Wouter}. When the metric spaces are proper (for example in the Riemannian context), a cobounded action is equivalent to a cocompact action. However, there are natural actions of certain groups on nonproper spaces. For example, when $\Ga$ is relatively hyperbolic, it naturally acts on a hyperbolic space which may not be proper. Another example is the mapping class group of a surface which acts on the (nonproper) curve complex.

We can further simplify our notions by pulling back the metric from the metric space to the group. Let $\Ga$ acts isometrically on a metric space $(X,d)$. By choosing any base point $o\in X$, we may identify $\Ga$ with its orbit $\Ga \cdot o\subset X$, which is also a metric space viewed as a subspace of $(X,d)$. Equivalently, it gives rise to a left invariant metric $d_X$ on $\Ga$, given by
\[d_X(\ga_1,\ga_2)=d(\ga_1 \cdot o, \ga_2 \cdot o)\quad \forall \ga_1,\ga_2\in \Ga.\]
Thus, via the left translation, $\Ga$ acts isometrically on $(\Ga,d_X)$, and it is clear that the marked length spectrum of this action coincides with that of the action on $(X,d)$. Moreover, such a pull back metric $d_X$ satisfies the following two additional properties:

\begin{enumerate}
	\item $d_X$ is quasi-isometric to the word metric of $\Ga$ if the action is cocompact: this is due to \v{S}varc-Milnor lemma. Also, the group $\Gamma$ is finitely generated if the action is properly discontinuous.
	\item $d_X$ is \emph{roughly geodesic} (See the precise definition as follows): This is because $(\Ga,d_X)$ can be identified with $\Ga \cdot o\subset X$, so any pair of points $\ga_1 \cdot o,\ga_2 \cdot o$ can be connected by a geodesic in $X$, and all points on the geodesic are always uniformly close to some orbit point since the action is cobounded.
\end{enumerate}
We recall a definition given by Bonk and Schramm.
\begin{definition}[\cite{BonkSchramm00}]Let $(X,d)$ be a metric space and let $\delta\ge 0$. We say the metric $d$ is \emph{$\delta$-roughly geodesic} metric if for every $x,y\in X$ there is a $(1,\delta)$-quasigeodesic from $x$ to $y$. A metric is said \emph{roughly geodesic} if there is $\delta\ge 0$ such that it is $\delta$-roughly geodesic.
\end{definition}

Inspired by the observations, our Question \ref{ques:main} can be reformulated (at least for proper metric spaces) as follows (cf. Burago and Margulis \cite[Problem session]{Oberwolfach06}):

\begin{question}\label{ques:rephrase}
	Let $\Ga$ be a finitely generated group, and $d_1, d_2$ be two roughly geodesic left invariant metrics that are quasi-isometric to a word metric. If the actions under $d_1, d_2$ have the same marked length spectrum, then are $d_1, d_2$ roughly equal, that is, $|d_1-d_2|\leq C$ for some constant $C$?
\end{question}

The main result of this paper is to give an affirmative answer to the question when $\Gamma$ is a relatively hyperbolic group. It is easy to see that the converse statement is also true: if two metrics are roughly equal, then their marked length spectra are the same.

\begin{theorem}\label{thm:main} Let $\Gamma$ be a finitely generated relatively hyperbolic group with respect to a family of infinite finitely generated proper peripheral subgroups $\{H_1,..., H_n\}$. Let $d_1$ and $d_2$ be two roughly geodesic left invariant metrics that are quasi-isometric to a proper word metric $d$. If the actions under $d_1, d_2$ have the same marked length spectrum, then $d_1, d_2$ are roughly equal.
\end{theorem}

We would like to point out that the roughly geodesic property of the left-invariant metrics is necessary for the theorem. We consider the example where $\Gamma=F_2=<a,b>$, the free group with two generators $a$ and $b$. On $\Gamma$ we let $d_1$ be the word metric defined by the generating set $S=\{a,a^{-1}, b, b^{-1}\}$. We also consider another left invariant metric $d_2$ defined by $d_2(\gamma_1,\gamma_2)=d_1(\gamma_1,\gamma_2)+\sqrt{d_1(\gamma_1,\gamma_2)}$, for every $\gamma_1, \gamma_2\in \Gamma$. We observe that  $d_2$ is $(2,1)$-quasi-isometric with $d_1$. The group $\Gamma$ is nonelementarily hyperbolic and thus is relatively hyperbolic, for example, relatively to the cyclic subgroup $<a>$. The two metrics  $d_1$ and $d_2$ have the same marked length spectrum, however they are not roughly equal. The result fails simply because $d_2$ is \emph{not} a roughly geodesic metric.

Regarding the history, Question \ref{ques:rephrase} was previously known to be true when
\begin{enumerate}
\item $\Gamma$ is abelian by Burago \cite{Burago94},
\item $\Gamma$ is a $3$-dimensional integer Heisenberg group by Krat \cite{Krat99},
\item $\Ga$ is finitely generated Gromov hyperbolic by Furman \cite{Furman02},
\item $\Gamma$ is a reductive Lie groups and the metrics $d_1, d_2$ are word metrics, by Abels-Margulis \cite{AbelsMargulis04},
\item $\Gamma$ is Gromov hyperbolic (a different proof from Furman's) or relatively hyperbolic groups with toral peripheral subgroups by Fujiwara \cite{Fujiwara15}.
\end{enumerate}
Our result further generalizes Fujiwara's as we do not have additional assumptions on the peripheral subgroups. When $\Gamma$ is non-elementary hyperbolic, we can treat $\Gamma$ as a relatively hyperbolic group with respect to an infinite cyclic subgroup, thus our result also recovers the hyperbolic case. Moreover, our proof is different from both Fujiwara and Furman's arguments. 

It is also worth pointing out that Question \ref{ques:rephrase} is known to be \emph{false} in general. Breuillard \cite[Section 8.3]{Breuillard14} gave a counterexample in the case $\Gamma = H^3\times \mathbb Z$ where $H^3$ is the integer Heisenberg group of dimension $3$.
\subsection*{Strategy of the proof}
The main technical part of the proof is to establish the genericity of certain elements in $\Ga$ with a \emph{good periodicity property}. These are elements $g$ such that for every $n\in \mathbb N$, geodesics from 1 to $g^n$ periodically passing through a uniform neigborhood of $g^i$ for each $0\le i \le n$, where the uniform constant depends on the metrics only and is independent of $n$. One key observation is that: if an element $g$ has good periodicity property and it has the same translation lengths with respect to two metrics $d_1$ and $d_2$, then $d_1(1,g)$ and $d_2(1,g)$ must be uniformly close. So the rigidity result holds at least for elements with good periodicity property. Furthermore, we show that for a relatively hyperbolic group, elements with good periodicity property is generic in an appropriate sense, and can be obtained from perturbations of any given element. Thus by triangle inequality, the rigidity result holds for all elements, that is, $d_1$ and $d_2$ are roughly equal.

Compared to the previous works of \cite{Burago94, Krat99, Fujiwara15}, our notion of good periodicity property is new and is the main difference. In \cite{Burago94, Krat99, Fujiwara15}, they showed that for all/generic elements $g\in \Gamma$, and for all $n\in \mathbb N$, there exists a geodesic from 1 to $g^n$ that are uniformly (does not depends on $g$ and $n$) close to the union of consecutive geodesic segments from $g^{i-1}$ to $g^i$, for every $1\le i\le n$. However, this is a much stronger property. For a general relatively hyperbolic group, we do not expect to have many such elements since the behaviors of geodesics along peripheral subgroups are mysterious. The good periodicity property we introduced is a good replacement and we are able to show for relatively hyperbolic groups, generic elements satisfy this property. The idea is to add a suitable prefix and suffix to any given element so that the resulting element initially travels transversely to all peripheral subsets and eventually stays in a particular peripheral subset. Then using the geometry of relatively hyperbolic group, we show this element has good periodicity property. This is done in Section \ref{sec:good_elt}.\\

It is natural to further ask for which classes of groups our theorem still holds. Our proof, and also the argument of Fujiwara \cite{Fujiwara15}, use certain hyperbolic geometry of the groups, and we believe it is possible to extend to a broader context in which the groups share similar hyperbolic features. On the other hand, the cautionary example of Breuillard indicates that such feature of hyperbolicity may be necessary. We end this section by the following question:

\begin{question}\label{ques:open} Does our theorem hold for products of hyperbolic groups, mapping class groups, or more generally hierarchically hyperbolic groups?
\end{question}

\subsection*{Acknowledgments:} We would like to thank Jason Behrstock, Igor Belegradek, Bob Bell, Cornelia Drutu, Mitul Islam, Jean Lafont, Ben Schmidt and Ralf Spatzier for useful discussions about relatively hyperbolic groups and the marked length spectrum rigidity. The authors would like to thank the University of Michigan and Michigan State University respectively for their hospitalities while the work was done.

\section{On relatively hyperbolic groups}\label{sec:recap}
Gromov in \cite{Gromov87} emphasized the parallelism between the notion of hyperbolicity in geometric group theory and manifolds of negative sectional curvature. The fundamental group of a closed negatively curved manifold has such hyperbolicity which is now known as Gromov hyperbolic (or $\delta$-hyperbolic). He also suggested that there should be a suitable notion of \emph{relative hyperbolicity} that serves as an analogous parallelism to include the fundamental group of finite volume negatively curved manifolds or more generally manifolds of non-positive sectional curvature.

Since the first attempt was given by Farb \cite{Farb98}, there has been many analogous definitions and equivalent characterizations due to the work of Bowditch, Drutu--Sapir, Osin, Groves--Manning, and Sisto \cite{Bowditch12, DrutuSapir05, Osin06, GrovesManning08, Sisto13}. In this paper, we use the notion of \emph{relative hyperbolicity} in the sense of \emph{strongly relative hyperbolicity} defined by Drutu and Sapir \cite[Definition 8.4]{DrutuSapir05}.

There are many well-known examples and constructions of relatively hyperbolic groups.
Fujiwara, in his work \cite{Fujiwara15}, gave several examples of relatively hyperbolic groups with toral peripheral subgroups. Since our theorem further generalizes to arbitrary relatively hyperbolic groups, we list here a few more examples of relatively hyperbolic groups whose peripheral subgroups are not toral.

\begin{enumerate}
\item If $G$ and $H$ are two infinite groups then $G*H$ is a relatively hyperbolic group with respect to the family $\{G,H\}$. If $G$ and $H$ are not hyperbolic and not toral then $G*H$ is an example of a relatively hyperbolic group with non toral peripheral subgroups.
\item Let $n$ be an integer that is at least 2. Non-uniform lattices in the following rank one Lie groups $\SU(n,1)$, $\Sp(n,1)$, or $F_4^{-20}$ are examples of relatively hyperbolic groups with nilpotent peripheral subgroups.
\item	One can construct closed smooth manifolds with relatively hyperbolic fundamental group using the relative hyperbolization \cite{Belegradek07}: For any compact manifold with boundary (e.g. a hyperbolic manifold with truncated cusps), after taking the relative hyperbolization, one obtains a new smooth manifold with boundary whose fundamental group is relatively hyperbolic. Then by taking the double and glue along the boundary, one constructs a closed manifold whose fundamental group is still relatively hyperbolic.
\end{enumerate}

\section{Elements with good periodicity property}\label{sec:good_elt}

For this section, we consider a finitely generated group $\Ga$, which is relatively hyperbolic with respect to a finite family $\{H_1,...,H_n\}$ of finitely generated, infinite, proper peripheral subgroups. We note that we do not lose any generality by assuming all peripheral subgroups are infinite as if any peripheral subgroup is finite, we can remove it from the list of peripheral subgroups while still keeping the relative hyperbolicity. 

It is shown in \cite[Theorem 9.1]{DrutuSapir05} that the Cayley graph $X$ of $\Ga$ is asymptotically tree-graded with respect to the corresponding peripheral subsets $\mathcal P=\{gH_i:g\in \Ga, 1\leq i\leq n\}$. Moreover, Sisto \cite[Theorem 2.14]{Sisto13} proved that $(X,\mathcal P)$ is an almost projection system (see definition below). Since the metric and the sets of peripheral subsets are (left) $\Gamma$-invariant, the system of projections is clearly $\Gamma$-equivariant. In this draft, we will mainly use Sisto's characterization. We recall here some of the definitions and results that we will heavily use in the proofs later. 

Denote $d$ the Cayley metric on $X$, then $(X,d)$ is a complete geodesic space where $\Ga$ acts isometrically. For simplicity, we identify $\Ga$ as the vertex subset in $X$, and denote $|g|=d(1,g)$, the norm measured in the word metric.

\begin{definition}(\cite[Definition 1.12]{Sisto13}) For every $x\in X$ and for every peripheral subset $P\in \mathcal P$, the projection $\pi_P(x)$ is defined as the set of points in $P$ whose distances to $x$ is less than $d(x,P)+1$.
\end{definition}
We note that when $X$ is a Cayley graph of $\Gamma$ and $\mathcal P$ is the collection of all left cosets of peripheral subgroups, then the projection $\pi_P(x)$ is precisely \emph{the set of points in $P$ whose distance equal $d(x,P)$}.

\begin{definition}(\cite[Definition 2.1]{Sisto13})\label{def:projection}
	We say a family of maps $\Pi=\{\pi_P:X\rightarrow P\}$ is an almost projection system if there exists a constant $C>0$ such that for any $P\in \mathcal P$, the following properities hold:
	\begin{enumerate}
		\item $\forall x\in X, \forall p\in P$, $d(x,p)\geq d(x,\pi_P(x))+d(\pi_P(x),p)-C$,
		\item $\forall x\in X$ with $d(x,P)=d$, $\op{diam}(\pi_P(B_d(x)))\leq C$,
		\item $\forall P\neq Q\in \mathcal P$, $\op{diam}(\pi_P(Q))\leq C$.
	\end{enumerate}
\end{definition}
For our purpose, we will only use the property $(3)$ in the above. As a special case, we have $\op{diam}(\pi_P(x))\leq C$ for any $x\in X$ and $P\in \mathcal P$.

\begin{lemma}\label{lem:proj along geod}
	Let $(X,\mathcal P)$ be an almost projection system. Given $P\in \mathcal P$ and $x\in X$. If $z\in\pi_P(x)$ is any point and $y$ is any point on a geodesic segment connecting $x$ and $z$, then $z\in \pi_P(y)$.
\end{lemma}

\begin{proof}
 First we notice by triangle inequality that $d(x,P)\leq d(x,y)+d(y,P)$. Since $z\in \pi_P(x)$, we have $d(x,z)<d(x,P)+1$. Since $y$ is a point on a geodesic segment connecting $x$ and $z$, we have $d(x,z)=d(x,y)+d(y,z)$. Now combining all the (in)equalities, we obtain $d(y,z)<d(y,P)+1$. Thus by definition, $z\in \pi_P(y)$.
\end{proof}
Besides, we will also need the following two geometric properties of this almost projection system.

\begin{lemma}\label{lem:sisto2.3}
	\cite[Lemma 2.3]{Sisto13} For every $x,y\in \Gamma$ and for every peripheral $P\in \mathcal P$,
	$$d(\pi_P(x),\pi_P(y))<d(x,y)+6C,$$
	where $C$ is the constant in Definition \ref{def:projection}.
\end{lemma}
In what follows and later, we use $\Nbhd_R(A)$ to denote the $R$-neighborhood of a set $A$, that is the set of all elements whose distance to the set $A$ is less than $R$.
\begin{lemma}\label{lem:sisto1.15}
	\cite[Lemma 1.15]{Sisto13} There exists $L$ so that if $d(\pi_P(x),\pi_P(y))\geq L$, then all $(K_0,C_0)$-quasi geodesics connecting $x,y$ intersect $\Nbhd_R(\pi_P(x))$ and $\Nbhd_R(\pi_P(y))$ where $R=R(K_0,C_0)$.
\end{lemma}

For the convenience, we introduce the following definitions.

\begin{definition}
	Let $T>0$ be a constant and $H$ be a peripheral subgroup. We say an element $g\in \Gamma$ 
	\begin{enumerate}
	\item has \emph{$T$-short head in $H$} if $d(\pi_H(g), 1)<T$, and 
	\item it has \emph{$T$-long tail in $H$} if $d(\pi_{gH}(1), g)>T$.
	\end{enumerate}	
\end{definition}

\begin{remark}
	Note by the $\Ga$-invariance of the metric, $g^{-1}$ has $T$-short head in $H$ if and only if  $d(\pi_{gH}(1),g)<T$, in particular, $g$ does not have $T$-long tail in $H$.
\end{remark}

From now on, we let $R=R(1,0)$ be the constant from Lemma \ref{lem:sisto1.15}. We fix a constant $R_0$ such that 
\begin{equation}\label{eq:R_0}
	R_0> \max\{R+10C, L+3C\},
\end{equation}
where $C, R, L$ are defined as above. The following proposition shows that given an arbitrary element in $\Ga$, we can always perturb by adding a prefix so that the new element has short head property.

\begin{proposition}\label{prop:perturb-start} For any $g\in \Ga$ and any peripheral subgroup $H$,

there exists $k\in \Gamma$ which satisfies
\begin{enumerate}
\item $|k|\le 1$,
\item $kg$ has $R_0$-short head in $H$.
\end{enumerate}
\end{proposition}

\begin{proof}
	We divide into the following two cases.\\
	{\bf Case 1}: when $g$ already has $R_0$-short hand in $H$. That is $d(\pi_H(g),1)< R_0$ for the given peripheral subgroup $H$. Then we can simply choose $k=1$, and the proposition holds automatically.\\
	{\bf Case 2}: when $g$ does not have $R_0$-short hand in $H$. That is $d(\pi_H(g),1)\ge  R_0$ for the given peripheral subgroup $H$. 
	We choose $k\in \Gamma\backslash H$ to be a generator of $\Gamma$ so that $|k|=1$. Such a $k$ always exists since $H$ is a proper subgroup of $\Ga$. We claim that $kg$ has $R_0$-short head in $H$. In other words, let $x\in \pi_H(kg)$ be an arbitrary element that achieves $d(1,x)=d(1,\pi_H(kg))$, we need to show $d(1,x)< R_0$.
	
	Since $|k|=1$, we have both $d(1,k)=d(k,H)$ and $d(1,k^{-1})=d(k^{-1},H)$, and it follows that $1\in \pi_H(k)$ and $1\in \pi_H(k^{-1})$. The latter containment implies $k\in \pi_{kH}(1)\subset \pi_{kH}(H)$ by the left invariance of the metric. Since $\text{diam}(\pi_{kH}(H))<C$ according to Definition \ref{def:projection}, we know $d(k,\pi_{kH}(x))<C$. By triangle inequality, 
	\begin{align*}&d(\pi_{kH}(x),\pi_{kH}(kg))\\&\ge d(k,\pi_{kH}(kg))-d(k,\pi_{kH}(x))-\op{diam}(\pi_{kH}(kg))-\op{diam}(\pi_{kH}(x))\\
		&\ge d(k,\pi_{kH}(kg))-d(k,\pi_{kH}(x))-2C
		\\ &\ge d(k,\pi_{kH}(kg))-3C\\ &=d(1,\pi_{H}(g))-3C\\ &\ge R_0-3C,
	\end{align*}
where the second inequality uses Definition \ref{def:projection}. By the choice of $R_0$, we have that $R_0-3C>L$ where $L$ is the constant from Lemma \ref{lem:sisto1.15}. It follows by this lemma that any geodesic from $kg$ to $x$ passes through a point $y$ in $R$-neighborhood of $\pi_{kH}(x)$ (recall $R=R(1,0)$). Since $d(k,\pi_{kH}(x))<C$, the point $y$ is contained in $(R+C)$-neighborhood of $k$, that is $d(y,k)\le R+C$. On the other hand, since the geodesic from $kg$ to $x\in \pi_H(kg)$ passes through $y$, by Lemma \ref{lem:proj along geod} the point $x$ is in $\pi_H(y)$. Therefore, we have
	\begin{align*}
		d(\pi_H(kg),1)=d(x,1)&\le d(\pi_H(y),1)+C\\
		&\leq d(\pi_H(y),\pi_H(k))+d(\pi_H(k),1)+2C+C\\	
		&\le (R+C)+6C+0+3C.\\
		&\leq R+10C,
	\end{align*}
where the first inequality uses the fact that $x\in \pi_H(y)$ and also Definition \ref{def:projection}, the second inequality follows from the triangle inequality and the third inequality uses Lemma \ref{lem:sisto2.3}. With the choice of $R_0$ such that $R_0> R+10C$, we have that $d(\pi_H(kg),1)< R_0$.
\end{proof}

In the next proposition, we show that the above element can be further perturbed by adding a suffix to obtain an element having the long tail property, while keeping the short head property.

\begin{proposition}\label{prop:adding-end}Let $g\in G$ such that $d(1,g)>R_0+1$. Then for every $C_3\geq \max\{1,C\}$, if there is a peripheral subgroup $H$ such that $g$ has $R_0$-short head, then  
	 there exists a perturbation $g'=gh$ such that
	\begin{enumerate}
		\item $|h|\le 4C_3$,
		\item $g'$ has $C_3$-long tail in $H$, and
		\item $g'$ has $(R_0+C)$-short head in $H$.
	\end{enumerate}
\end{proposition}

\begin{proof}If $d(\pi_{gH}(1), g)> C_3$, we then just choose $h=1$. Thus, we assume that $d(\pi_{gH}(1), g)\le  C_3$. Since $H$ is an infinite subgroup, we can choose $h\in H$ such that $3C_3< |h|\leq 4C_3$, and let $g'=gh$. By triangle inequality,
	\begin{align*}
	d(\pi_{g'H}(1), g')=d(\pi_{gH}(1), g')&\ge d(g, g')-d(\pi_{gH}(1), g)-C\\ 
	&> 3C_3-C_3-C\\
	&\geq C_3.
	\end{align*}
	It remains to check that the new element $g'$ still has short head property. That is to check $d(\pi_H(g'),1)<R_0+C$. For this, we first note that $g\notin H$. Otherwise $g\in H$ implies that $\pi_H(g)=B_g(1)$, the ball radius 1 around $g$. Thus by triangle inequality $d(1,g)<d(1,\pi_H(g))+1$. Since $d(1,g)>R_0+1$, we have $d(1,\pi_H(g))>R_0$, which contradicts with the assumption that $g$ has $R_0$-short head in $H$.
	
	Therefore $\op{diam}(\pi_H(g'H))=\op{diam}(\pi_H(gH))<C$ by Definition \ref{def:projection}. It follows by triangle inequality that $d(\pi_H(g'),1)\le d(\pi_H(g),1)+C<R_0+C$.
\end{proof}

Elements with short head and long tail in a peripheral subgroup will be shown later to satisfy \emph{the good periodicity property}, and Proposition \ref{prop:perturb-start} and Proposition \ref{prop:adding-end} simply descibes how such an element can be obtained from a perturbation.

Next we show that if an element $g$ has short head and long tail in some peripheral subgroup $H$, then its power $g^n$ also has short head and long tail in $H$, with possibly worse constants, but they are uniform on $n$.

\begin{proposition}\label{prop:good-element} For any $\delta>0$, there exist constants $C_1,C_2, C_3>0$ explicitly given by
	\[C_1=\delta+R+9C\]
	\[ C_2=C_1+L+3C\]
and
\[C_3=C_2+R+10C+\delta+2,\]
such that the following holds: for every $g\in G$, if there exists a peripheral subgroup $H$ such that $g$ has $\delta$-short head and $C_3$-long tail in $H$, then for any $n\in \mathbb N$, we have
	\begin{enumerate}
		\item $g^n$ has $C_1$-short head in $H$.
		\item $g^n$ has $C_2$-long tail in $H$. 		
	\end{enumerate}
\end{proposition}

\begin{proof} We first note that the element $g$ which satisfies the condition cannot belong to the peripheral subgroup $H$. Otherwise, $\pi_H(g)$ consists of elements in $H$ whose distance to $g$ are at most $1$, so by triangle inequality $d(\pi_H(g),1)\geq d(g,1)-1$. On the other hand, since $g$ has $C_3$-long tail in $H$, we have $d(\pi_{gH}(1), g)>C_3$. It follows similarly that $d(\pi_{gH}(1), g)<d(1,g)+1$. Thus, combining with the previous inequalities, we have
	\[d(g,\pi_{gH}(1))>C_3-2>\delta,\]
which contradicts the $\delta$-short head assumption of $g$.

For every $n\in \mathbb N$, we let $P(n)$ be the statement that $(1)$ and $(2)$ hold for $n$. We prove by induction that $P(n)$ holds for every $n\in \mathbb N$.\\
{\bf Base case $n=1$}: This follows immediately from the assumption of $g$ and the choices that $C_1>\delta$ and $C_3> C_2$.\\
{\bf Induction hypothesis}:	Suppose $P(n)$ holds for all $n\leq k$ (where $k\in \mathbb N$).\\
{\bf Induction step}: We need to show $P(k+1)$ holds. We prove $(1)$ and $(2)$ separately, that is, $d(\pi_{H}(g^{k+1}),1)<C_1$ and $d(\pi_{g^{k+1}H}(1),g^{k+1})>C_2$ respectively.
	
Proof of $(1)$: Let $x\in \pi_{H}(g^{k+1})\subset H$ be an arbitrary element. We denote $\ga$ a geodesic segment from $g^{k+1}$ to $x$. By induction hypothesis, we have 
\[d(\pi_{gH}(g^{k+1}), g)=d(\pi_{H}(g^k),1)<C_1,\] and also $d(\pi_{gH}(1),g)>C_2$. Since $g\notin H$, by Defninition \ref{def:projection} we have $\op{diam}(\pi_{gH}(H))<C$ . So by the triangle inequality, we have 
	\[d(\pi_{gH}(x), g)\geq d(\pi_{gH}(1), g)-\op{diam}(\pi_{gH}(H))> C_2-C.\]
	Thus,
	\begin{align*}
		&d(\pi_{gH}(x), \pi_{gH}(g^{k+1}))\\&\geq d(\pi_{gH}(x), g)-d(\pi_{gH}(g^{k+1}), g)-\op{diam}(\pi_{gH}(x))-\op{diam}(\pi_{gH}(g^{k+1}))\\
		&\geq C_2-C-C_1-C-C\\
		&\geq L,
	\end{align*}
by the choice of $C_1, C_2$. Apply Lemma \ref{lem:sisto1.15} on $\ga$, we see that $\ga$ intersects with $\Nbhd_R(\pi_{gH}(x))$, and if we let $z\in \ga\cap \Nbhd_R(\pi_{gH}(x))$ and let $z'\in \pi_{gH}(x)$ such that $d(z,z')<R$, then by Lemma \ref{lem:sisto2.3}, we have
\begin{align*}
	d(\pi_H(z),\pi_H(z'))&<d(z,z')+6C\\
	&\leq R+6C.
\end{align*}
Since $\ga$ is a geodesic connecting $g^{k+1}$ with $x$ and $x\in\pi_{H}(g^{k+1})$, we have $x\in\pi_{H}(z)$ by Lemma \ref{lem:proj along geod}.
	On the other hand, since both $z'$ and $g$ belongs to the peripheral subset $gH$, we see that $\pi_H(z')$ and $\pi_H(g)$ are both contained in the set $\pi_H(gH)$ whose diameter is bounded by $C$. Thus we have
	\begin{align*}
		d(1,\pi_H(z'))&\leq d(1,\pi_H(g))+C<\delta+C,
	\end{align*}
	where the last inequality uses the assumption that $g$ has $\delta$-short head in $H$.
	
	Finally, we can estimate
	\begin{align*}
		d(1,x)&\leq d(1,\pi_H(z'))+d(\pi_H(z'),\pi_H(z))+d(\pi_H(z),x)+2C\\
		&< (\delta+C)+(R+6C)+0+2C\\
		&=C_1,
	\end{align*}
by the choice of $C_1$. Thus, $(1)$ of $P(k+1)$ holds.
	
Proof of $(2)$: Let $y\in \pi_{g^{k+1}H}(1)$ be arbitrary. We denote $\gamma'$ a geodesic path from $1$ to $y$.	By induction hypothesis, we have 
	\begin{align*}
		d(\pi_{g^kH}(g^{k+1}),g^k)=d(\pi_{H}(g),1)<\delta,
	\end{align*}
	and also $d(\pi_{g^kH}(1), g^k)>C_2$. Note that \[\op{diam}(\pi_{g^kH}(g^{k+1}H))=\op{diam}(\pi_{H}(gH))<C\]
	since $g\notin H$. So $d(\pi_{g^kH}(y),\pi_{g^kH}(g^{k+1}))\le \op{diam}(\pi_{g^kH}(g^{k+1}H)\le C$. Thus by the triangle inequality, we have 
	\begin{align*}
		d(\pi_{g^{k}H}(y), \pi_{g^k H}(1))&\geq d(\pi_{g^k H}(1),g^k)-d(\pi_{g^kH}(g^{k+1}), g^k)-\op{diam}(\pi_{g^kH}(g^{k+1}))\\
		&> C_2-\delta-C\\
		&> L.
	\end{align*}
	By Lemma \ref{lem:sisto1.15}, we have $\ga'\cap \Nbhd_R(\pi_{g^kH}(y))\neq \emptyset$. We choose $w\in \pi_{g^kH}(y)$ and $w'\in\ga'\cap \Nbhd_R(\pi_{g^kH}(y))$ such that $d(w,w')<R$. Then by triangle inequality, we have
	\begin{align*}
		d(w',g^k)&\leq d(w',\pi_{g^kH}(y))+\op{diam}(\pi_{g^kH}(g^{k+1}H))+d(\pi_{g^kH}(g^{k+1}), g^k)\\
		&<d(w',w)+C+d(\pi_H(g),1)\\
		&<R+C+\delta,
	\end{align*}
	where the second inequality follows from the fact that $w\in \pi_{g^kH}(y)$, by Lemma \ref{lem:proj along geod}.
	On the other hand, since $\ga'$ is a geodesic and $y\in \pi_{g^{k+1}H}(1)$, by Lemma \ref{lem:proj along geod} we have $y\in \pi_{g^{k+1}H}(w')$. It follows by Lemma \ref{lem:sisto2.3} that
	\begin{align*}
		d(y,\pi_{g^{k+1}H}(g^k))&\le d(y, \pi_{g^{k+1}H}(w'))+ d(\pi_{g^{k+1}H}(w'),\pi_{g^{k+1}H}(g^k))\\
		&+\op{diam}(\pi_{g^{k+1}H}(w'))+\op{diam}(\pi_{g^{k+1}H}(g^k))\\
		&<0+ d(w',g^k)+6C+2C\\
		&<R+C+\delta+8C\\
		&=R+9C+\delta.
	\end{align*}
	Note that $d(\pi_{g^{k+1}H}(g^k),g^{k+1})=d(\pi_{gH}(1),g)>C_3$ by the assumption. Thus,
	\begin{align*}
		d(y,g^{k+1})&\geq d(\pi_{g^{k+1}H}(g^k),g^{k+1})-d(y,\pi_{g^{k+1}H}(g^k))-\op{diam}(\pi_{g^{k+1}H}(g^k))\\
		&>C_3-(R+9C+\delta)-C\\
		&>C_2.
	\end{align*}
Hence $(2)$ of $P(k+1)$ holds.
	
	Therefore, by induction, the proposition holds. 
\end{proof}

We now fix all constants so that they depend only on the projection system, or equivalently a finitely generated word metric on $\Gamma$. We first fix $R_0$ by choosing
 \begin{equation}\label{eq:fix R_0}
 R_0=R+L+10C,
 \end{equation}
where we recall that $R=R(1,0)$ and $L$ are constants from Lemma \ref{lem:sisto1.15}. In particular, it satisfies inequality \eqref{eq:R_0}. Besides, we also fix $\delta=R_0+C$ as in the following proposition. We emphasize that by doing so, all constants including $R, L, C, C_1, C_2$, and $C_3$ are now determined and they only depend on the projection system. We summarize the above discussions into the following proposition, which will be the main technical preparation of our theorem.

\begin{proposition}\label{prop:perturb-end} There exist $C_1,C_2,C_3$ explicitly given by
	\[C_1=\delta+R+9C\]
	\[ C_2=C_1+L+3C\]
	\[C_3=C_2+R+10C+\delta+2,\]
	where 
\[\delta=R_0+C=R+L+11C,\]
with the following properties: for every element $g_0\in \Gamma$ with $|g_0|>R_0+2$, there exist an element $g=kg_0h\in \Gamma$ and a peripheral subgroup $H$ such that
\begin{enumerate}
\item $|k|\le 1$.
\item $|h|\le 4C_3$.
\item for every $n\in \mathbb N$, $g^n$ has $C_1$-short head and $C_2$-long tail in $H$.
\end{enumerate}
\end{proposition}

\begin{proof}

Given any element $g_0\in \Ga$, by Proposition \ref{prop:perturb-start}, there exist $k\in \Gamma$ and a peripheral subgroup $H$ such that $|k|\leq 1$ and $kg_0$ has $R_0$-short head in $H$. Since $|g_0|>R_0+2$, by triangle inequality we have $|kg_0|\geq |g_0|-|k|>R_0+1$.  By applying Proposition \ref{prop:adding-end} on the element $kg_0$, there exists $h\in \Gamma$ such that $|h|<4C_3$ and $kg_0h$ has $(R_0+C)$-short head, $C_3$-long tail in $H$. Apply Proposition \ref{prop:good-element} to the element $g=kg_0h$ and on the scale $\delta=R_0+C$, we conclude that $g^n$ has $C_1$-short head and $C_2$-long tail in $H$ for any $n\in \mathbb N$.
\end{proof}

One significance of obtaining such an element is that all its iterated powers belong to a uniform neighborhood of a geodesic. This is what we were referring to as \emph{the good periodicity property} in the introduction, and we will show in the next proposition. Thus, together with the above proposition, we show that up to appropriate perturbations, essentially all elements in a relatively hyperbolic group will have good periodicity property.

\begin{proposition}\label{prop:periodicMorse}Let $g$ be an element such that for every $n\in \mathbb N$, $g^n$ has $C_1$-short head and $C_2$-long tail in a peripheral subgroup $H$, where $C_1$ and $C_2$ are constants from Proposition \ref{prop:perturb-end}. For every $\tau\ge 1$ and $\eta\ge 0$, there is a constant $\kappa>0$, depending only on $\tau, \eta$ and $\Ga$, such that for every $n\in \mathbb N$ and for every $(\tau,\eta)$-quasigeodesic $\gamma$ from 1 to $g^n$, $g^i$ is in $\kappa$-neighborhood of $\gamma$ for each $0\leq i\leq n$.
\end{proposition}

\begin{proof}
For every $0< i< n$, since $g^i$ has $C_2$-long tail in $H$, we know $d(\pi_{g^iH}(1),g^i)>C_2$. Since $g^{n-i}$ has $C_1$-short head in $H$, we have
\[d(g^i,\pi_{g^iH}(g^n))=d(1,\pi_H(g^{n-i}))<C_1.\]
Thus, by triangle inequality, we have
\begin{align*}
	&d(\pi_{g^iH}(1),\pi_{g^iH}(g^n))\\&\ge d(\pi_{g^iH}(1),g^i)-d(g^i,\pi_{g^iH}(g^n))-\op{diam}(\pi_{g^iH}(g^n))-\op{diam}(\pi_{g^iH}(1))\\
	&> C_2-C_1-2C\\
	&\geq L
\end{align*}
by the choice of $C_1, C_2$. By Lemma \ref{lem:sisto1.15}, there exists a constant $R(\tau,\eta)$ such that $\gamma$ intersects the neighborhood $\Nbhd_{R(\tau,\eta)}(\pi_{g^iH}(g^n))$. On the other hand, since $d(\pi_{g^iH}(g^n),g^i)=d(\pi_H(g^{n-i}), 1)<C_1$, it follows by the triangle inequality that $g^i$ is in $(R(\tau,\eta)+C_1+C)$-neighborhood of $\gamma$. 

For $i=0$ or $i=n$, $g^i$ is in the $\eta$-neighborhood of $\ga$. Hence if we choose $\kappa = R(\tau,\eta)+C_1 + C+\eta$, the proposition follows.
\end{proof}

\section{Rigidity}

In this section, we assume that $d_1$ and $d_2$ are two left invariant roughly geodesic metrics on a finitely generated relatively hyperbolic group $\Gamma$. Furthermore, assume that $d_1$ and $d_2$ have the same marked length spectrum. We let $\Delta=|d_1-d_2|$. Suppose that both metrics are $(L_0,C_0)$-quasi-isometric to a word metric $d$ defined by a finite generating set.

Fujiwara proved the following almost additivity for the difference between two metrics.

\begin{lemma}\cite[Lemma 2.3]{Fujiwara15}\label{lem:additivity} For $i=1,2$, let $\gamma_i$ be a $(1,\delta)$-roughly geodesic from $x$ to $y$ with respect to $d_i$ metric. For every $\kappa>0$ there exists $\bar\kappa>0$ such that the following holds: if $z\in \gamma_1$ is a point such that there exists $z'\in \gamma_2$ with $d_2(z,z') \le \kappa$, then $|\Delta(x, z)+\Delta(z, y) - \Delta(x, y)|<\bar\kappa$.
\end{lemma}

We are now ready to give a proof of our main theorem.

\begin{proof}[Proof of Theorem \ref{thm:main}]
Suppose $\tau\ge 1$ and $\eta\ge 0$  are constants such that $d_1$ and $d_2$ are $(\tau,\eta)$-quasi-isometric to a word metric $d$. The marked length spectrum assumption implies that for every $g\in \Gamma$, we have the following sub-linear growth $\Delta(1,g^n)=o(n)$.

Suppose, for the sake of contradiction, that $d_1$ and $d_2$ are not roughly equal. Thus for every $\mathcal E>0$, there is an element (hence infinitely many elements) $g'\in \Gamma$ such that $\Delta (1, g')>\mathcal E$. By possibly avoiding finitely many elements, we may assume $|g'|>R_0+2$ where we recall $|\cdot |$ is the norm measured in word metric $d$ and $R_0$ is the constant chosen in \eqref{eq:fix R_0}. By Proposition \ref{prop:perturb-end}, there is a peripheral subgroup $H$, elements $k, h\in \Gamma$, and constants $C_1,C_2,C_3$ such that
\begin{enumerate}
\item $|d(1,k)|\le 1$,
\item $|d(1,h)|\le 4C_3$,
\item For every $n$, the element $g^n$ has $C_1$-short head and $C_2$-long tail in $H$, where $g=kg'h$.
\end{enumerate}

Since $d_1$ and $d_2$ are roughly geodesic, there is $\delta_0$ such that for any $x,y\in \Gamma$ are $(1,\delta_0)$ quasi-geodesics from $x$ to $y$ with respect to $d_1$ and $d_2$. As $d_1$ and $d$ are $(\tau,\eta)$ quasi-isometric, any $(1,\delta_0)$-quasi-geodesic in $d_1$ metric is a $(\tau,\tau\delta_0+\eta)$-quasi-geodesic in the word metric $d$.

By Proposition \ref{prop:periodicMorse}, there exists $\kappa$ for every $n\in \mathbb N$, any $(1,\delta_0)$-quasi-geodesic $\gamma_n$ in $d_1$ metric from $1$ to $g^n$ passes through $\kappa$-neighborhoods of $g^i$ for every $0\le i\le n$. Similarly, for every $(1,\delta_0)$-quasi-geodesic $\beta_n$ from 1 to $g^n$ in $d_2$ metric,  $\beta_n$ passes through $\kappa$-neighborhoods of $g^i$ for every $0\le i\le n$. By Lemma \ref{lem:additivity}, there exists $\bar\kappa\ge 0$ such that
$$|\Delta(1, g)+\Delta(g, g^2)+\dots +\Delta(g^{n-1}, g^n) - \Delta(1, g^n)|<n\bar\kappa.$$
By left-invariance of $d_1, d_2$, and thus of $\Delta$, this is equivalent with 
$$|n\Delta(1, g) - \Delta(1, g^n)|<n\bar\kappa.$$
Therefore, $\Delta(1,g^n)>n(\Delta(1,g)-\bar\kappa)$. We note that $C_1,C_2,C_3$ and thus $\kappa$, $\bar\kappa$ depend only on $\delta_0$, $\tau$, $\eta$, and the metric $d$, but they are independent of the constant $\mathcal E$. Since $g=kg'h$ with  $|d(1,k)|\le 1$ and $|d(1,h)|\le 4C_3$, using triangle inequalities we have that
$$\Delta(1,g)>\Delta(1,g')- 2(\tau+\eta + \tau \cdot 4C_3 +\eta).$$
Hence if we choose $\mathcal E>2(\tau+\eta + \tau \cdot 4C_3 +\eta)+\bar\kappa + 1$ then $\Delta(1,g^n)>n$ for every $n\in \mathbb N$. This contradicts with the sub-linear growth $\Delta(1,g^n)=o(n)$. 

Therefore $d_1$ and $d_2$ are roughly equal.
\end{proof}

As an immediate corollary of our Theorem \ref{thm:main}, we have the following:

\begin{corollary}
	Let $M$ be a closed smooth manifold whose fundamental group is relatively hyperbolic. If $g_1$ and $g_2$ are two non-positively curved Riemannian metrics with the same marked length spectrum, then 
	there exists a constant $C\geq 0$ and a $\Ga$-equivariant $(1,C)$-quasi-isometry between their universal covers $(\widetilde M,\widetilde g_1)$ and $(\widetilde M,\widetilde g_2)$.
\end{corollary}	

\bibliographystyle{alpha}
\bibliography{MLS}

\begin{thebibliography}{LSvL17}

\bibitem[AM04]{AbelsMargulis04}
Herbert Abels and Gregory Margulis.
\newblock Coarsely geodesic metrics on reductive groups.
\newblock In {\em Modern dynamical systems and applications}, pages 163--183.
  Cambridge Univ. Press, Cambridge, 2004.

\bibitem[BCG95]{BCG95}
G.~Besson, G.~Courtois, and S.~Gallot.
\newblock Entropies et rigidit\'{e}s des espaces localement sym\'{e}triques de
  courbure strictement n\'{e}gative.
\newblock {\em Geom. Funct. Anal.}, 5(5):731--799, 1995.

\bibitem[Bel07]{Belegradek07}
Igor Belegradek.
\newblock Aspherical manifolds with relatively hyperbolic fundamental groups.
\newblock {\em Geom. Dedicata}, 129:119--144, 2007.

\bibitem[BK85]{Burns-Katok1985}
K.~Burns and A.~Katok.
\newblock Manifolds with nonpositive curvature.
\newblock {\em Ergodic Theory Dynam. Systems}, 5(2):307--317, 1985.

\bibitem[Bow12]{Bowditch12}
B.~H. Bowditch.
\newblock Relatively hyperbolic groups.
\newblock {\em Internat. J. Algebra Comput.}, 22(3):1250016, 66, 2012.

\bibitem[Bre14]{Breuillard14}
Emmanuel Breuillard.
\newblock Geometry of locally compact groups of polynomial growth and shape of
  large balls.
\newblock {\em Groups Geom. Dyn.}, 8(3):669--732, 2014.

\bibitem[BS00]{BonkSchramm00}
M.~Bonk and O.~Schramm.
\newblock Embeddings of {G}romov hyperbolic spaces.
\newblock {\em Geom. Funct. Anal.}, 10(2):266--306, 2000.

\bibitem[Bur92]{Burago94}
D.~Yu. Burago.
\newblock Periodic metrics.
\newblock In {\em Representation theory and dynamical systems}, volume~9 of
  {\em Adv. Soviet Math.}, pages 205--210. Amer. Math. Soc., Providence, RI,
  1992.

\bibitem[Cro90]{Croke}
Christopher~B. Croke.
\newblock Rigidity for surfaces of nonpositive curvature.
\newblock {\em Comment. Math. Helv.}, 65(1):150--169, 1990.

\bibitem[DS05]{DrutuSapir05}
Cornelia Dru\c{t}u and Mark Sapir.
\newblock Tree-graded spaces and asymptotic cones of groups.
\newblock {\em Topology}, 44(5):959--1058, 2005.
\newblock With an appendix by Denis Osin and Mark Sapir.

\bibitem[Far98]{Farb98}
B.~Farb.
\newblock Relatively hyperbolic groups.
\newblock {\em Geom. Funct. Anal.}, 8(5):810--840, 1998.

\bibitem[Fuj15]{Fujiwara15}
Koji Fujiwara.
\newblock Asymptotically isometric metrics on relatively hyperbolic groups and
  marked length spectrum.
\newblock {\em J. Topol. Anal.}, 7(2):345--359, 2015.

\bibitem[Fur02]{Furman02}
Alex Furman.
\newblock Coarse-geometric perspective on negatively curved manifolds and
  groups.
\newblock In {\em Rigidity in dynamics and geometry ({C}ambridge, 2000)}, pages
  149--166. Springer, Berlin, 2002.

\bibitem[GL19]{Guillarmou19}
Colin Guillarmou and Thibault Lefeuvre.
\newblock The marked length spectrum of {A}nosov manifolds.
\newblock {\em Ann. of Math. (2)}, 190(1):321--344, 2019.

\bibitem[GM08]{GrovesManning08}
Daniel Groves and Jason~Fox Manning.
\newblock Dehn filling in relatively hyperbolic groups.
\newblock {\em Israel J. Math.}, 168:317--429, 2008.

\bibitem[Gro87]{Gromov87}
M.~Gromov.
\newblock Hyperbolic groups.
\newblock In {\em Essays in group theory}, volume~8 of {\em Math. Sci. Res.
  Inst. Publ.}, pages 75--263. Springer, New York, 1987.

\bibitem[Ham99]{Hamenstaedt99}
U.~Hamenst{{\"a}}dt.
\newblock Cocycles, symplectic structures and intersection.
\newblock {\em Geom. Funct. Anal.}, 9(1):90--140, 1999.

\bibitem[KLS16]{Kar-Lafont-Schmidt}
Aditi Kar, Jean-Fran\c{c}ois Lafont, and Benjamin Schmidt.
\newblock Rigidity of almost-isometric universal covers.
\newblock {\em Indiana Univ. Math. J.}, 65(2):585--613, 2016.

\bibitem[Kra99]{Krat99}
S.~A. Krat.
\newblock Asymptotic properties of the {H}eisenberg group.
\newblock {\em Zap. Nauchn. Sem. S.-Peterburg. Otdel. Mat. Inst. Steklov.
  (POMI)}, 261(Geom. i Topol. 4):125--154, 268, 1999.

\bibitem[LSvL17]{Lafont-Schmidt-Wouter}
Jean-Fran\c{c}ois Lafont, Benjamin Schmidt, and Wouter van Limbeek.
\newblock Quasicircle boundaries and exotic almost-isometries.
\newblock {\em Ann. Inst. Fourier (Grenoble)}, 67(2):863--877, 2017.

\bibitem[Obe06]{Oberwolfach06}
Geometric group theory, hyperbolic dynamics and symplectic geometry.
\newblock volume~3, pages 1991--2057. 2006.
\newblock Abstracts from the workshop held July 23--29, 2006, Organized by
  Gerhard Knieper, Leonid Polterovich and Leonid Potyagailo, Oberwolfach
  Reports. Vol. 3, no. 3.

\bibitem[Osi06]{Osin06}
Denis~V. Osin.
\newblock Relatively hyperbolic groups: intrinsic geometry, algebraic
  properties, and algorithmic problems.
\newblock {\em Mem. Amer. Math. Soc.}, 179(843):vi+100, 2006.

\bibitem[Ota90]{Otal}
Jean-Pierre Otal.
\newblock Sur les longueurs des g{\'e}od{\'e}siques d'une m{\'e}trique {\`a}
  courbure n{\'e}gative dans le disque.
\newblock {\em Comment. Math. Helv.}, 65(2):334--347, 1990.

\bibitem[Sis13]{Sisto13}
Alessandro Sisto.
\newblock Projections and relative hyperbolicity.
\newblock {\em Enseign. Math. (2)}, 59(1-2):165--181, 2013.

\end{thebibliography}
%\bibliographystyle{myalpha}
%\bibliography{allbib}

\end{document}